\documentclass{birkmult}[10pt]
\usepackage{mathrsfs}
\usepackage[centertags]{amsmath}
\usepackage{amsfonts}
\usepackage{amsthm}
\usepackage{color}

\usepackage[active]{srcltx}

\usepackage{newsymbol}
\let\emptyset \undefined
\let\ge       \undefined
\let\le       \undefined
\newsymbol\le          1336  \let\leq\le
\newsymbol\ge          133E  
\newsymbol\emptyset    203F
\newsymbol\notle       230A
\newsymbol\notge       230B

\allowdisplaybreaks

\usepackage{fancyhdr}

\newcommand\dela[1]{\Green{{\small{14.12.2009}}}}

\theoremstyle{plain}

\input colordvi
\newcommand{\n}{\|}
\,

\newcommand\E{{\mathbb E}}

\newcommand\R{{\mathbb R}}
 \newtheorem{thm}{Theorem}[section]
 
 \newtheorem{lem}[thm]{Lemma}
 
 \theoremstyle{definition}
 
 \theoremstyle{remark}
 
 \newtheorem{ex}[thm]{Example}
 \numberwithin{equation}{section}  

\def\d{\text{\rm{d}}}
\newcommand{\F}{\mathscr{F}}
\renewcommand{\P}{\mathbb{P}}
\newcommand{\dom}{{\mathsf D}}
\newcommand{\e}{\varepsilon}

\begin{document}
%
%
%
%
%
%
%
%
%
\title[Stochastic convolutions in $2$-smooth Banach spaces]
 {A maximal inequality for stochastic convolutions in $2$-smooth Banach spaces}

\author{Jan van Neerven and Jiahui Zhu} 
\address{Delft Institute of Applied Mathematics\\
Delft University of Technology\\
P.O. Box 5031, 2600 GA Delft\\
The Netherlands} 
\email{J.M.A.M.vanNeerven/Jiahui.Zhu@TUDelft.nl}



\thanks{The authors are supported by VICI subsidy 639.033.604
of the Netherlands Organisation for Scientific Research (NWO)}

\date{today}

\begin{abstract}
Let $(e^{tA})_{t\ge 0}$ be a $C_0$-contraction 
semigroup on a $2$-smooth Banach space $E$,
let $(W_t)_{t\ge 0}$ be a cylindrical Brownian motion in a Hilbert space $H$, and let
$(g_t)_{t\ge 0}$ be a progressively measurable process with values in the space $\gamma(H,E)$
of all $\gamma$-radonifying operators from $H$ to $E$.
We prove that for all $0<p<\infty $ there exists a constant $C$, depending only 
on $p$ and $E$, such that for all $T\ge 0$ we have
$$\E \sup_{0\le t\le T} \Big\n   \int_0^t e^{(t-s)A}g_s\,\d W_s\Big\n ^p 
\le C \E \Big(\int_0^T \n  g_t\n _{\gamma(H,E)}^2 \,\d t\Big)^\frac{p}{2}.
$$
For $p\ge 2$ the proof is based on the observation that $\psi(x) = \n x\n^p$
is Fr\'echet differentiable and its derivative satisfies 
the Lipschitz estimate $\n \psi'(x) - \psi'(y)\n \le C(\n x\n + \n y\n)^{p-2}\n x-y\n$;
the extension to $0<p<2$ proceeds via Lenglart's inequality.
\end{abstract}

\maketitle

\section{Introduction}

Let $(e^{tA})_{t\ge 0}$ be a $C_0$-contraction 
semigroup on a $2$-smooth Banach space $E$ and
let $(W_t)_{t\ge 0}$ be a cylindrical Brownian motion in a Hilbert space $H$. Let
$(g_t)_{t\ge 0}$ be a progressively measurable process with values in the space $\gamma(H,E)$ 
of all $\gamma$-radonifying operators from $H$ to $E$ satisfying
$$ \int_0^T \n g_t\n _{\gamma(H,E)}^2\,\d t < \infty \ \ \hbox{$\P$-almost surely}$$  
for all $T\ge 0$. As is well known (see \cite{[Dett], [Neidhardt], [Ondrejat]}), 
under these assumptions the stochastic convolution process
$$ X_t = \int_0^t e^{(t-s)A}g_s\,\d W_s, \quad t\ge 0,$$  
is well-defined in $E$ and provides the unique mild solution of the stochastic initial value problem
$$  \,\d X_t = AX_t \,\d t + g_t\,\d W_t, \quad X_0 = 0.$$
In order to obtain
the existence of a continuous version of this process, one usually 
proves a maximal estimate of the form
\begin{align}\label{eq:max} \E \sup_{0\le t\le T} \n  X_t\n ^p 
\le C^p \E \Big(\int_0^T \n  g_t\n _{\gamma(H,E)}^2 \,\d t\Big)^\frac{p}{2}.
\end{align}
The first such estimate was obtained by Kotelenez \cite{[Kot1], [Kot2]} for $C_0$-contraction semigroups 
on Hilbert spaces $E$ and exponent $p=2$. 
Tubaro \cite{[Tub]} extended this result to exponents $p\ge 2$ by a different method of proof
which applies It\^o's formula to the $C^2$-mapping $x\mapsto \n  x\n ^p$. The case $p\in (0,2)$
was covered subsequently by Ichikawa \cite{[Ich]}.
A very simple proof, still for $C_0$-contraction semigroups on Hilbert spaces, 
which works for all $p\in (0,\infty)$, was obtained recently by Hausenblas and Seidler \cite{[HauSei]}.
It is based on the Sz.-Nagy dilation theorem, which is used to reduce the problem to the corresponding
problem for $C_0$-contraction groups. Then, by using the group property, the maximal estimate follows from 
Burkholder's inequality. This proof shows, moreover, that the constant $C$ in \eqref{eq:max}
may be taken equal to the constant appearing in Burkholder's inequality. In particular,
this constant depends only on $p$.

The maximal inequality 
\eqref{eq:max} has been extended by Brze\'zniak and Peszat \cite{[BrzPes]} 
to $C_0$-contraction semigroups on Banach spaces $E$ with the property that, for some $p\in [2,\infty)$,
$x\mapsto \n x\n ^p$ is twice continuously Fr\'echet differentiable and the first and second 
Fr\'echet derivatives are bounded by constant multiples of $\n x\n^{p-1}$ and $\n x\n^{p-2}$, respectively. 
Examples of spaces with this property, which we shall call $(C_p^2)$, are the spaces $L^q(\mu)$
for $q\in [p,\infty)$. Any $(C_p^2)$ space is $2$-smooth (the definition is recalled in Section \ref{sec:smooth}), 
but the converse doesn't hold:

\begin{ex} Let $F$ be a Banach space.
The space $\ell^2(F)$ is $2$-smooth whenever $F$ is $2$-smooth \cite[Proposition 17]{[Fig]}. 
On the other hand, the norm of $\ell^2(F)$ is twice continuously Fr\'echet differentiable away from the origin
if and only if $F$ is a Hilbert space
\cite[Theorem 3.9]{[LeoSun]}. Thus, for $q\in (2,\infty)$,  $\ell^2(\ell^q)$ and $\ell^2(L^q(0,1))$ 
are examples of $2$-smooth Banach spaces which fail property $(C_p^2$) for all $p\in [2,\infty)$.
\end{ex}

To the best of our knowledge,
the general problem of proving the maximal estimate \eqref{eq:max} 
for $C_0$-contraction semigroups on $2$-smooth Banach space remains open. The present paper aims 
to fill this gap:

\begin{thm}\label{thm:main} 
Let $(e^{tA})_{t\ge 0}$ be a $C_0$-contraction semigroup on a $2$-smooth Banach space $E$,
let $(W_t)_{t\ge 0}$ be a cylindrical Brownian motion in a Hilbert space $H$, and
let $(g_t)_{t\ge 0}$ be a progressively measurable process in $\gamma(H,E)$.
If
$$\int_0^T \n  g_t\n _{\gamma(H,E)}^2 \,\d t < \infty \quad \P\hbox{-almost surely,}$$
then the stochastic convolution process 
$ X_t = \int_0^t e^{(t-s)A} g_s\, \d W_s$ is well-defined and has a continuous version.
Moreover, for all $0<p<\infty$ there exists a constant $C$, depending only 
on $p$ and $E$, such that 
$$\E \sup_{0\le t\le T} \n  X_t\n ^p 
\le C^p \E \Big(\int_0^T \n  g_t\n _{\gamma(H,E)}^2 \,\d t\Big)^\frac{p}{2}.
$$
\end{thm}

For $p\ge 2$, the proof of Theorem \ref{thm:main} is based on a version of It\^o's formula 
(Theorem \ref{thm-main-1}) which 
exploits the fact (proved in Lemma \ref{mian-lem}) that in $2$-smooth Banach spaces the function
$\psi(x) = \n x\n^p$
is Fr\'echet differentiable and satisfies 
the Lipschitz estimate $$\n \psi'(x) - \psi'(y)\n \le C(\n x\n + \n y\n)^{p-2}\n x-y\n.$$
The extension to exponents $0<p<2$ is obtained by applying Lenglart's inequality 
(see \eqref{eq:Lenglart}).

We conclude this introduction with a brief discussion of some developments of the inequality
\eqref{eq:max} into different directions in the literature. 
Seidler \cite{[Sei]} has proved the inequality \eqref{eq:max} with optimal 
constant $C = O(\sqrt{p})$ as $p\to\infty$ 
for positive $C_0$-contraction semigroups on the ($2$-smooth) space $E = L^q(\mu)$,
$q\ge 2$. He also proved that the same result holds if the assumption 
`$e^{tA}$ is a positive contraction semigroup' is replaced by `$-A$ has a bounded $H^\infty$-calculus
of angle strictly less than $\frac12\pi$'. The latter result was subsequently extended
by Veraar and Weis \cite{[VerWei]} to arbitrary UMD spaces $E$ with type $2$. 
In the same paper, still under the assumption that $-A$ has a bounded $H^\infty$-calculus
of angle strictly less than $\frac12\pi$, the following stronger estimate is
obtained for UMD spaces $E$ with Pisier's property $(\alpha)$:
\begin{align}\label{eq:max-gamma} \E \sup_{0\le t\le T} \n  X_t\n ^p 
\le C^p  \E \n g\n _{\gamma(L^2(0,T;H),E)}^p
\end{align}
with a constant $C$ depending only on $p$ and $E$.
If, in addition, $E$ has type $2$, then the mapping $f\otimes (h\otimes x)\mapsto
(f\otimes h)\otimes x$ extends to a continuous embedding $L^2(0,T;\gamma(H,E))\hookrightarrow 
\gamma(L^2(0,T;H),E)$ and \eqref{eq:max-gamma} implies \eqref{eq:max}.

Let us finally mention that, for $p>2$, a weaker version of \eqref{eq:max} 
for arbitrary $C_0$-semigroups on Hilbert spaces has been obtained 
by Da Prato and Zabczyk \cite{[DZ]}. Using the factorisation method they proved that
$$ \E \sup_{0\le t\le T} \n  X_t\n ^p 
\le C^p \E \int_0^T \n  g_t\n _{\gamma(H,E)}^p \,\d t
$$
with a constant $C$ depending on $p$, $E$, and $T$.
The proof extends {\em verbatim} to $C_0$-semigroups on martingale type $2$ spaces. This 
relates to the above results for $2$-smooth spaces through a theorem 
of Pisier \cite[Theorem 3.1]{[P1]}, which states that a Banach space has 
martingale type $p$ if and only if it is $p$-smooth.

\section{The Fr\'echet derivative of $\n \cdot \n ^p$}\label{sec:smooth}

Let $1< q\le 2$.
A Banach space $E$ is {\em $q$-smooth} if the modulus of smoothness 
$$ \rho_{\n \cdot\n }(t) = \sup\Big\{\tfrac12 (\n  x+ty\n  + \n  x-ty\n ) -1\, :\ \n  x\n 
=\n  y\n  = 1 \Big\}$$
satisfies $\rho_{\n \cdot\n }(t)\leq C t^q$ for all $t>0$. 

It is known (see \cite[Theorem 3.1]{[P1]}) 
that $E$ is $q$-smooth if and only if there exists a constant $K\ge 1$ such that 
for all $x,y\in E$,
\begin{align}\label{eq:K}
	 \n  x+y\n ^q+\n  x-y\n ^q\leq 2\n  x\n ^q+K\n  y\n ^q.
\end{align}   

 \begin{lem}\label{mian-lem}
Let $E$
be a Banach space and let $1<q\le 2$ be given. For $p\ge q$ set $\psi_p(x):= \n x \n ^p$.
	\begin{enumerate}
\item $E$ is $q$-smooth if and only if the Fr\'echet derivative of $\psi_q$ is 
	globally $(q-1)$-H\"{o}lder continuous on $E$.    
\item If $E$ is $q$-smooth, then for $p>q$ the Fr\'echet derivative of $\psi_p$ is locally $(q-1)$-H\"{o}lder continuous on $E$.
\end{enumerate}
Moreover, for all $p\ge q$ and $x,y\in E$ we have
\begin{align}\label{Ap-lem-1-ep-1}
 \n \psi_p'(x) - \psi_p'(y)\n \le C(\n x\n +\n y\n )^{p-q}\n x-y\n ^{q-1},
\end{align} 
where $C$ depends only on $p$, $q$ and $E$.
\end{lem}

\begin{proof} 
If the Fr\'echet derivative of $\psi_q$ is $(q-1)$-H\"{o}lder continuous on $E$, 
then by the mean value theorem we can find $0\leq \theta,\rho\leq 1$ such that 
for all $x,y\in E$,
\begin{align*} 
\n x+y\n^q+\n x-y\n^q -2\n x\n^q &
 = (\n x+y\n^q - \n x\n^q) + (\n x-y\n^q - \n x\n^q)
\\ & \leq \n \psi_q'(x+\theta y)-\psi_q'(x-\rho y)\n\, \n y\n \\  
&\leq L\n (x+\theta y)-(x-\rho y)\n ^{q-1}\n y\n \leq 2^{q-1}L\n y\n ^q.
\end{align*}
Hence the Banach space $E$ is $q$-smooth.  

Suppose now that the norm of $E$ is $q$-smooth. Then for all $x,y\in E$ 
with $\n x\n ,\n y\n =1$ and all $t>0$ we have
\begin{align}\label{ep-lem-1}
         \n x+ty\n +\n x-ty\n -2\n x\n \leq K \n ty\n ^q.
\end{align}
Thus $$\lim_{t\rightarrow0}\frac{\n x+ty\n +\n x-ty\n -2\n x\n }{\n ty\n }=0,$$ 
which by \cite[Lemma I.1.3]{[Deville]} 
means that $\n \cdot \n $ is  Fr\'{e}chet differentiable on the unit sphere.
Hence, by homogeneity, $\n \cdot \n $ is Fr\'{e}chet 
differentiable on 
$E\backslash\{0\}$. Let us denote by $f_x$ its Fr\'{e}chet  derivative at the point $x\not=0$.

We begin by showing the $(q-1)$-H\"{o}lder continuity of $x\mapsto f_x$ on the unit sphere of $E$,
following the argument of \cite[Lemma V.3.5]{[Deville]}. We fix $x\not=y\in E$ such 
that $\n x\n ,\n y\n =1$ and $h\in E$ with $\n h\n =\n x-y\n $ and $x-y+h\not=0$. 
Since the norm $\n \cdot \n $ is a convex function, 
\begin{align*}            
f_y(x-y)\leq \n x\n -\n y\n. 
\end{align*}  
Similarly, we have
\begin{align*}
         f_x(h)&\leq \n x+h\n -\n x\n, \qquad 
          f_y(y-x-h)\leq \n 2y-x-h\n -\n y\n .
\end{align*} 
By using above inequalities and the linearity of the function $f_x$, we have
\begin{align*}
f_x(h)-f_y(h)\leq \n x+h\n -\n x\n -f_y(h)&=\n x+h\n -\n y\n -f_y(x+h-y)+\n y\n -\n x\n +f_y(x-y)\\
& \leq \n x+h\n -\n y\n -f_y(x+h-y)\\
&=\n x+h\n -\n y\n +f_y(y-x-h)\\
&\leq \n x+h\n +\n 2y-x-h\n -2\n y\n \\
&=\Big\n y+\n x+h-y\n \cdot\frac{x+h-y}{\n x+h-y\n }\Big\n
\\ & \qquad  +\Big\n y-\n x+h-y\n \cdot
\frac{x+h-y}{\n x+h-y\n }\Big\n-2\n y\n
\\ &\leq K\n x+h-y\n ^q  \leq K(\n x-y\n +\n h\n )^q = 2^qK\n x-y\n ^q,
\end{align*}                                         
where we also used \eqref{ep-lem-1}.
Since the roles of $x$ and $y$ may be reversed in this inequality, this implies 
\begin{align*}
   \n f_x-f_y\n =\sup_{\n h\n =\n x-y\n }\frac{|f_x(h)-f_y(h)|}{\n x-y\n } \leq 2^qK\n x-y\n ^{q-1}
\end{align*}                                                                 
This proves the $(q-1)$-H\"{o}lder continuity of the norm $\n \cdot \n $ on 
the unit sphere. 

We proceed with the proof of \eqref{Ap-lem-1-ep-1}; the $(q-1)$-H\"older continuity
of $\psi_q$ as well as the local $(p-1)$-H\"older continuity
of $\psi_p$ follow from it. For all $x,y\in E$ with $x\neq 0$ 
and $y\neq 0$ we have $\psi_p'(x)=p\n x\n ^{p-1}f_x$. 

It is easy to check that $f_x=f_{\frac{x}{\n x\n }}$ and 
$\n f_x\n =1$. 
Following once more the argument of \cite[Lemma V.3.5]{[Deville]}, this gives
\begin{equation}
 \label{eq:holder-p}
\begin{aligned}
        \n \psi_p'(x)-\psi_p'(y)\n &=p\big\n \n x\n ^{p-1}f_x-\n y\n ^{p-1}
f_y\big\n \\
& \leq p\Big\n \n x\n ^{p-1}(f_{\frac{x}{\n x\n }}-f_{\frac{y}{\n y\n }})
\Big\n +p\Big\n (\n x\n ^{p-1}-\n y\n ^{p-1}) f_{\frac{y}{\n y\n }}\Big\n\\ 
& \leq p2^q K\n x\n ^{p-1}\Big\n \frac{x}{\n x\n }-\frac{y}{\n y\n }\Big\n ^{q-1}
+p\Big| \n x\n ^{p-1}-\n y\n ^{p-1}\Big| \\
& \leq p2^qK \n x\n ^{p-q} \n y\n ^{1-q}\Big\n  x\n y\n -y\n x\n \Big\n ^{q-1}
+p\Big| \n x\n ^{p-1}-\n y\n ^{p-1}\Big| \\
& = p2^q K \n x\n ^{p-q}\n y\n ^{1-q}\Big\n \n y\n (x-y)+y(\n y\n -\n x\n )\Big\n ^{q-1} + 
p\Big| \n x\n ^{p-1}-\n y\n ^{p-1}\Big| \\
& \le p2^q K \n x\n ^{p-q}\n y\n ^{1-q}(2\n y\n \n x-y\n) ^{q-1} + 
p\Big| \n x\n ^{p-1}-\n y\n ^{p-1}\Big| \\
& = p2^{2q-1}K\n x\n ^{p-q}\n x-y\n ^{q-1} 
+p\Big| \n x\n ^{p-1}-\n y\n ^{p-1}\Big|. 
\end{aligned}
\end{equation} 
If $q\le p\leq 2$, then by the inequality $|t^r-s^r|\leq |t-s|^r$, 
valid for $0<r\leq 1$ and $s,t\in[0,\infty)$, we have
\begin{align*}
         \big|\n x\n ^{p-1}-\n y\n ^{p-1}\big|\leq \big|\n x\n -\n y\n \big|^{p-1}
\leq \n x-y\n ^{p-1}
\leq (\n x\n +\n y\n )^{p-q}\n x-y\n ^{q-1}.
\end{align*} 
If $p>2$, by applying the mean value theorem, for some $\theta\in [0,1]$ we have 
\begin{align*}
    \big|\n x\n ^{p-1}-\n y\n ^{p-1}\big|&=(p-1)\Big\n \n \theta x+(1-\theta)y\n ^{p-2}
f_{\theta x+(1-\theta )y}(x-y)\Big\n \\
&\leq (p-1)(\n x\n +\n y\n )^{p-2}\n x-y\n \\
&\leq (p-1)(\n x\n +\n y\n )^{p-2}(\n x\n +\n y\n )^{2-q}\n x-y\n ^{q-1}
\\ &=(p-1)(\n x\n +\n y\n )^{p-q}\n x-y\n ^{q-1}.
\end{align*} 
Also, since $\psi_p'(0)=0$, for $y\not=0$ we have
\begin{align*}
         \n \psi_p'(0)-\psi_p'(y)\n =p\n y\n ^{p-1} = p\n y\n ^{p-1}\Big\n \frac{y}{\n y\n} \Big\n ^{p-1}
\le p\n y\n ^{p-1}\Big\n \frac{y}{\n y\n} \Big\n ^{q-1} =  p\n y\n ^{p-q}\n y\n ^{q-1}.
\end{align*}               
\end{proof}
 
The above lemma will be combined with the next one, which gives a first order 
Taylor formula with 
a remainder term involving the first derivative only.  

\begin{lem}\label{lem:Taylor}
  Let $E$ and $F$ be Banach spaces, let $0<\alpha\le 1$,  and  let $\psi:E\rightarrow F$ 
be a Fr\'{e}chet differentiable function whose Fr\'{e}chet  
derivative $\psi': E\rightarrow \mathscr{L}(E,F)$ is locally 
$\alpha$-H\"{o}lder continuous.
 Then for all $x,y\in E$ we have
  \begin{align*}
         \psi(y)=\psi(x)+\psi'(x)(y-x)+R(x,y),       
  \end{align*} 
where
  \begin{align}\label{eq:Taylor}
   R(x,y)=\int_0^1(\psi'(x+r(y-x))(y-x)-\psi'(x)(y-x)) \,\d r.
    \end{align}
\end{lem}
\begin{proof}
     Pick $w\in E$ such that $\n w\n \leq 1$ and consider the function $f:\R\to F$ by
     \begin{align*}
            f(\theta) := \psi(x+\theta w).
     \end{align*}
For all $x^*\in F^*$, $\langle f', x^*\rangle $ is locally 
$\alpha$-H\"{o}lder continuous. To see this, note 
 that for $|\theta_1|,|\theta_2|\leq R$ and $\n  x\n \leq R$ we have  
$\n  x+\theta_1 w\n ,\n  x+\theta_2 w\n \leq 2R$, so by assumption there exists a constant 
$C_{2R}$ such that
     \begin{align*}
            |\langle f'(\theta_1)-f'(\theta_2), x^*\rangle| &=
| \langle\psi'(x+\theta_1w)w,x^*\rangle -\langle\psi'(x+\theta_2w)w,x^*\rangle |\\
            &\leq \n \psi'(x+\theta_1w)-\psi'(x+\theta_2w)\n \,\n  x^*\n  
             \leq C_{2R} |\theta_1-\theta_2|^\alpha\n  x^*\n .
            \end{align*}
Applying Taylor's formula and \cite[Lemma 1, Theorem 3]{[Ana]} to the 
function $\langle f,x^*\rangle $
we obtain
     \begin{align*}
          \langle f(t)-f(0),x^*\rangle = t\langle f'(0),x^*\rangle +\langle R_f(0,t), x^*\rangle ,
     \end{align*}
     where  $R_f(0,t)=\int_0^1 t(f'(rt)-f'(0))\,\d r$.
Now let $x,y\in E$ be given and set 
$t=\n y-x\n $ and $w=\frac{y-x}{\n  y-x\n }$.  
With these choices we obtain
        \begin{align*}
 \langle \psi(y),x^* \rangle- \langle \psi(x),x^*\rangle -
\langle \psi'(x)(y-x),x^*\rangle
&= \langle \psi(x+tw),x^* \rangle- 
\langle \psi(x),x^*\rangle - t\langle \psi('x)w,x^*\rangle   \\
          &= \langle  f(t)-f(0) -t f'(0), x^*\rangle \\
 & = \int_0^1 t \langle f'(rt)-f'(0),x^*\rangle \,\d r
\\ & = \int_0^1 \langle \psi'(x+r (y-x))(y-x)-\psi'(x)(y-x),x^*\rangle 
\,\d r.
    \end{align*} 
Since $x^*\in F^*$ was arbitrary, this proves the lemma.
    \end{proof}
 
\section{An It\^o formula for $\n \cdot \n ^p$}

From now on we shall always assume that $E$ is a $2$-smooth Banach space. We fix $T\ge 0$ and 
let $(\Omega,\F,\P)$ be a probability space with a filtration $(\F_t)_{t\in [0,T]}$. Let 
$H$ be a real Hilbert space, and denote by $\gamma(H,E)$ the Banach space of all
$\gamma$-radonifying operators from $H$ to $E$. 
We denote by $M([0,T];\gamma(H,E))$ the space of all progressively measurable
 processes $\xi:[0,T]\times\Omega\rightarrow \gamma(H,E)$ such that 
 \begin{align*}
 \int_0^T\n \xi_t\n ^2_{\gamma(H,E)}\,\d t<\infty \ \ \hbox{$\P$-almost surely}.
 \end{align*}
The space of all such $\xi$ which satisfy 
$$ \E \Big(\int_0^T\n \xi_t\n ^2_{\gamma(H,E)}\,\d t\Big)^\frac{p}{2}<\infty $$
is denoted by $M^p([0,T];\gamma(H,E))$, $0<p<\infty$.

On $(\Omega,\F,\P)$, let 
$(W_t)_{t\in [0,T]}$ be an $(\F_t)_{t\in [0,T]}$-cylindrical Brownian motion in $H$.  
For adapted simple processes $\xi\in M([0,T];\gamma(H,E))$
of the form 
 \begin{align*} 
	    \xi_t=\sum_{i=0}^{n-1}1_{(t_i,t_{i+1}]}(t)\otimes A_i,
\end{align*} 
where $\Pi=\{0=t_0<t_1<\cdots<t_n=T\}$ is a partition of the interval $[0,T]$
and the 
random variables $A_i$ are $\mathcal{F}_{t_i}$-measurable and take values in  
the space of all finite rank operators from $H$ to $E$, we define 
the random variable $ I( \xi)\in L^0(\Omega,\F_T;E)$ by
\begin{align*} 
I(\xi) := \sum_{i=0}^{n-1}A_{i}(W_{t_{i+1}}-W_{t_i})
\end{align*}
where $(h\otimes x)W_t := (W_t h)\otimes x$. 
It is well known that 
$$\E \n I(\xi)\n^2 \le C^2 \E \int_0^T \n \xi_t\n_{\gamma(H,E)}^2\,\d t,$$
where $C$ depends on $p$ and $E$ only. It follows that 
$I$ has a unique extension to a bounded linear operator
$M^2([0,T];\gamma(H,E))$ to $L^2(\Omega,\F_T;E)$. 
By a standard localisation argument, $I$ extends
continuous linear operator
from $M([0,T];\gamma(H,E))$ to $L^0(\Omega,\F_T;E)$. In what follows we write
\begin{align*}
  \int_0^t\xi_s \,\d W_s := I (1_{(0,t]}\xi), \quad t\in [0,T].
\end{align*} 
This stochastic integral has the following properties:
\begin{enumerate}
 \item
For all $\xi\in M([0,T];\gamma(H,E))$ the process $t\rightarrow \int_0^t\xi_s \,\d W_s$ 
is an $E$-valued continuous local martingale, which is a martingale if $\xi\in M^2([0,T];\gamma(H,E))$. 
\item For all $\xi\in M([0,T];\gamma(H,E))$ and stopping times $\tau$ with values in $[0,T]$,
\begin{align}\label{eq:stopping}
\int_0^\tau \xi_t\,\d W_t = \int_0^T 1_{[0,\tau]}(t)\xi_t \,\d W_t \quad \hbox{$\P$-almost surely}.
\end{align} 
\item For all $\xi\in M^2([0,T];\gamma(H,E))$ and $0\le u<t\le T$, 
 \begin{align}\label{Iso-ine}
     \E\Big(\Big\n \int_u^t\xi_s\,\d W_s\Big\n ^2|\F_u\Big)\leq
C\E\Big(\int_u^t\n \xi_s\n ^2_{\gamma(H,E)}\, \d s\,|\F_u\Big).
 \end{align} 
\item (Burkholder's inequality \cite{[Ass], [Dett]}) 
For all $0<p<\infty$ there exists a constant $C$, depending only on $p$ and $E$, such
that for all $\xi\in M^p([0,T];\gamma(H,E))$ and $t\in [0,T]$,
 \begin{align}\label{burk-ine}
    \E\sup_{s\in[0,t]}\Big\n \int_0^s\xi_u\,\d W_u\Big\n ^p\leq
C\E\Big(\int_0^t\n \xi_s\n ^2_{\gamma(H,E)}\,\d s\Big)^{\frac{p}{2}}.
 \end{align}
\end{enumerate}
   
An excellent survey of the theory of stochastic integration in $2$-smooth Banach spaces with 
complete proofs
is given in Ondrej\'at's thesis \cite{[Ondrejat]}, where also further references to the literature 
can be found.

In what follows we fix $p\ge 2$ and set $\psi(x):= \psi_p(x) = \n x\n ^p.$  
Since we assume that $E$ is $2$-smooth, this function is Fr\'echet differentiable.       
Following the notation of Lemma \ref{lem:Taylor} we set
$$ 
   R_\psi(x,y):=\int_0^1
(\psi'(x+r(y-x))(y-x)-\psi'(x)(y-x)) \,\d r.
$$ 
We have the following version of It\^o's formula.

\begin{thm}[It\^o formula]\label{thm-main-1} Let $E$ be a $2$-smooth Banach space 
and let $2\le p<\infty$.
Let $(a_t)_{t\in [0,T]}$ be an $E$-valued progressively measurable process
such that $$\E \Big(\int_0^T \n a_t\n \,\d t\Big)^p<\infty$$
and let 
$(g_t)_{t\in [0,T]}$ be a process in $M^p([0,T];\gamma(H,E))$.
Fix $x\in E$ and let $(X_t)_{t\in [0,T]}$ be given by
    \begin{align*}
   X_t=x+\int_0^ta_s\,\d s+\int_0^tg_s \,\d W_s.
\end{align*}
 The process  $s\mapsto \psi'(X_s) g_s$ is progressively measurable and belongs to
$M^1([0,T];H)$, and for all $t\in [0,T]$ we have
 \begin{align}\label{eq:Ito}
\psi(X_t)=\psi(x)+\int_0^t\psi'(X_s)(a_s)\,\d s+\int_0^t\psi'(X_s)(g_s)\,\d W_s+
\lim_{n\rightarrow\infty}\sum_{i=0}^{m(n)-1} R_\psi(X_{t_{i}^n\wedge
t},X_{t_{i+1}^n\wedge t})
\end{align}
with convergence in probability, 
for any sequence of partitions $\Pi_n=\{0=t_0^n<t_1^n<\cdots<t_{m(n)}^n=T\}$ whose meshes 
$\n \Pi_n\n :=\max_{0\leq i\leq m(n)-1}|t_{i+1}^n-t_i^n|$ 
tend to $0$ as $n\to \infty$. Moreover, there exists a constant $C$ and, for each $\e>0$, a constant 
$C_{\e}$, both independent of 
$a$ and $g$, such that
\begin{align}\label{eq:remainder}
  \E\liminf_{n\rightarrow\infty}\sum_{i=0}^{m(n)-1}
|R_\psi(X_{t_{i}^n\wedge t},X_{t_{i+1}^n\wedge t})| 
\le
\varepsilon C
\E \sup_{s\in [0,t]}\n X_s\n ^p
+C_{\e}
\E\Big(\int_0^t\n g_s\n_{\gamma(H,E)} ^2\,\d s\Big)^{\frac{p}{2}}.
\end{align}
\end{thm}   

The proof shows that we may take $C_\e = C'(\e^{1-\frac2p} +1)$ for some constant
$C'$ independent of 
$a$, $g$, and $\e$.

Before we start the proof of the theorem we state some lemmas.
The first is an immediate consequence of Burkholder's inequality \eqref{burk-ine}.

\begin{lem}\label{lem:0}
Under the assumptions of Theorem \ref{thm-main-1}
we have 
$$ \E \sup_{0\le t\le T} \n X_t\n ^p \le C \E \Big(\int_0^T \n a_s\n \,\d s\Big)^p
+ C \E \Big(\int_0^T \n g_s\n_{\gamma(H,E)} ^2\,\d s\Big)^\frac{p}{2}.$$
\end{lem}

\begin{lem}\label{lem:1a}
Under the assumptions of Theorem \ref{thm-main-1}, 
the process $t\mapsto \psi'(X_t) (g_t)$ is progressively measurable and 
belongs to $M^1([0,T];H)$.
\end{lem}
\begin{proof}
By the identity $\n \psi'(x)\n = p\n x\n^{p-1}$ and H\"older's inequality,
 \begin{align*} 
  \E\Big(\int_0^T \n \psi'(X_t)(g_t)\n_H ^2\,\d t\Big)^\frac{1}{2} 
 &\leq \E\Big(\int_0^T \n \psi'(X_t)\n ^2\n g_t\n ^2_{\gamma(H,E)}\,\d t\Big)^\frac{1}{2}\\
 &\leq \E \sup_{t\in[0,T]}\n X_t\n ^{p-1} \Big(\int_0^T\n g_t\n ^2_{\gamma(H,E)}\,\d s\Big)^\frac{1}{2}\\
 &\leq C \Big(\E \sup_{t\in[0,T]}\n X_t\n ^{p}\Big)^{\frac{p-1}{p}} 
\Big(\E \Big(\int_0^T\n g_t\n ^2_{\gamma(H,E)}   \,\d s\Big)^\frac{p}{2}\Big)^{\frac{1}{p}},
  \end{align*}
and the right-hand side is finite by the previous lemma.
The progressively measurability is clear.
\end{proof}

This lemma implies that
the stochastic integral in \eqref{eq:Ito} is well-defined.

\begin{lem}\label{lem:1b}
Let $0\le u\le t\le T$ be arbitrary and fixed.
Under the assumptions of Theorem \ref{thm-main-1}, 
the process $s\mapsto \psi'(X_u) (g_s)$ is progressively measurable and 
belongs to $M^1([0,T];H)$.
Moreover, $\P$-almost surely,  
$$ \psi'(X_{u})\int_u^t g_s\,\d W_s
= \int_u^t \psi'(X_u) (g_s)\,\d W_s.
$$
\end{lem}

\begin{proof} 
By similar estimates as in the previous lemma,
 \begin{align*} 
  \E\Big(\int_u^t \n \psi'(X_u)(g_s)\n_H ^2\,\d s\Big)^\frac{1}{2}
  \leq C (\E \n X_u\n ^{p})^{\frac{p-1}{p}} 
\Big(\E \Big(\int_u^t\n g_s\n ^2_{\gamma(H,E)}   \,\d s\Big)^\frac{p}{2}\Big)^{\frac{1}{p}}.
  \end{align*}
The progressively measurability is again clear.
To prove the identity we first assume that $g$ is a simple adapted process of the form
 \begin{align*} 
	    g_s=\sum_{i=0}^{n-1}1_{(t_i,t_{i+1}]}(s)A_i,
\end{align*} 
where $\Pi=\{u=t_0<t_1<\cdots<t_n=t\}$ is a partition of the interval $[0,T]$
and the 
random variables are $\mathcal{F}_{t_i}$-measurable and take values in  
the space of all finite rank operators from $H$ to $E$.
Then,
\begin{align*} 
 \psi'(X_u)\int_u^t g_s\,\d W_s&=\psi'(X_u)
\Big(\sum_{i=0}^{n-1}A_{i}(W_{t_{i+1}}-W_{t_i}) \Big) \\
   &= \sum_{i=0}^{n-1}\psi'(X_u) (A_{i}(W_{t_{i+1}}-W_{t_i}))
= \int_u^t\psi'(X_u) (g_s)\,\d W_s.
\end{align*}
For general progressively measurable $g \in L^p(\Omega;L^2([0,T];\gamma(H,E)))$,
the identity follows by a routine approximation argument.
\end{proof}

 \begin{proof}  [Proof of Theorem \ref{thm-main-1}]
The proof of the theorem proceeds in two steps. All constants occurring in the proof 
may depend on $E$
and $p$, even where this is not indicated explicitly, but not on $T$. 
The numerical value of the constants may change from line to line.

\smallskip 
{\em Step 1} -- 
Applying Lemma \ref{lem:Taylor}
to the function $\psi(x)=\n x\n ^p$ 
and the process $X$, we have, for every
$t\in[0,T]$,
\begin{align*}
\psi(X_t)-\psi(x)&=\sum_{i=0}^{m(n)-1}\Big(\psi(X_{t_{i+1}^n\wedge
t})-\psi(X_{t_i^n\wedge t})\Big)\\
        &=\sum_{i=0}^{m(n)-1}\psi'(X_{t_{i}^n\wedge
t})(X_{t_{i+1}^n\wedge t}-X_{t_{i}^n\wedge t}) 
           +\sum_{i=0}^{m(n)-1}
R_\psi(X_{t_{i}^n\wedge t},X_{t_{i+1}^n\wedge t}).
\end{align*}   
We shall prove the identity \eqref{eq:Ito} by showing that 
                \begin{align*}             
\lim_{n\rightarrow\infty}\sum_{i=0}^{m(n)-1}\psi'(X_{t_{i}^n\wedge
t})(X_{t_{i+1}^n\wedge t}-X_{t_{i}^n})                    
=\int_0^t\psi'(X_s)(a_s)\,\d s+\int_0^t\psi'(X_s)( g_s) \,\d W_s
                                    \end{align*}
with convergence in probability.  
In view of the definition of $X_t$, it is enough to show that 
                \begin{align*}
                       \lim_{n\rightarrow\infty} \Big|
\sum_{i=0}^{m(n)-1}\psi'(X_{t_i^n\wedge t})\Big{(}\int_{t_i^n\wedge
t}^{t_{i+1}^n\wedge t}a_s\,\d s\Big{)}-\int_0^t \psi'(X_s)(a_s)\,\d s     \Big|=0 \ \
\text{$\P$-almost surely}
                \end{align*} 
and
\begin{align}\label{ep-31-1}
 \lim_{n\rightarrow\infty} 
\sum_{i=0}^{m(n)-1}\psi'(X_{t_i^n\wedge t})\Big{(}\int_{t_i^n\wedge
t}^{t_{i+1}^n\wedge t} g_s\,\d W_s\Big{)}-
    \int_{0}^{t} \psi'(X_s)(g_s)\,\d W_s =0 \ \ \text{in probability}.
\end{align}  
By \eqref{Ap-lem-1-ep-1}, $\P$-almost surely we have
          \begin{align*}
               & \limsup_{n\rightarrow\infty} \Big|
\sum_{i=0}^{m(n)-1}\psi'(X_{t_i^n\wedge t})\Big{(}\int_{t_i^n\wedge
t}^{t_{i+1}^n\wedge t} a_s\Big{)}-\int_{0}^{t} \psi'(X_s)(a_s)\,\d s \Big|\\
                &\leq \limsup_{n\rightarrow\infty}
\sum_{i=0}^{m(n)-1}\Big|\int_{t_i^n\wedge t}^{t_{i+1}^n\wedge t}
(\psi'(X_{t_i^n\wedge t})-\psi'(X_s))(a_s)\,\d s \Big|\\
                &\leq C
\sup_{s\in[0,T]}\n X_s\n ^{p-2}\times \limsup_{n\rightarrow\infty}
\sum_{i=0}^{m(n)-1}\int_{t_i^n\wedge t}^{t_{i+1}^n\wedge t}
\n X_{t_i^n\wedge t}-X_s\n \,\n a_s\n \,\d s \\
                &\leq C
\sup_{s\in[0,T]}\n X_s\n ^{p-2}
\times \limsup_{n\rightarrow\infty} \Big(\sup_{0 \leq i\leq m(n)-1}
\sup_{s\in[t_i^n\wedge t,t_{i+1}^n\wedge t]}\n X_{t_i^n\wedge t}-X_s\n \Big)               
\times\Big{(}\sum_{i=0}^{m(n)-1}\int_{t_i^n\wedge t}^{t_{i+1}^n\wedge t}
\n a_s\n \,\d s\Big{)}\\
                &=0,
          \end{align*} 
where we used the continuity of the process $X$ in the last
line.           

Next, by Lemma \ref{lem:1b} and the inequalities \eqref{Iso-ine}  
and \eqref{Ap-lem-1-ep-1}, 
\begin{align*}  
\ & \sum_{i=0}^{m(n)-1}\psi'(X_{t_i^n\wedge
t})\Big{(}\int_{t_i^n\wedge t}^{t_{i+1}^n\wedge t} g_s\,\d W_s\Big{)}-
    \int_{0}^{t} \psi'(X_s)(g_s)\,\d W_s\\
& \qquad =\sum_{i=0}^{m(n)-1}\int_{t_i^n\wedge
t}^{t_{i+1}^n\wedge t}\psi'(X_{t_i^n\wedge t})( g_s)\,\d W_s-
    \int_{0}^{t} \psi'(X_s)(g_s)\,\d W_s\\
& \qquad = \int_0^t\sum_{i=0}^{m(n)-1}1_{(t_i^n,t_{i+1}^n]}(s)(\psi'(X_{t_i^n\wedge
t})-\psi'(X_s))(g_s)\,\d W_s.
\end{align*}
Recall that the localized stochastic integral is continuous from 
$M([0,t];\gamma(H,E)))$ into $L^0(\Omega,\F_t;E)$. Hence, in order 
to prove that the right-hand side converges to $0$ in probability, it suffices to prove
that
$$\lim_{n\to\infty} \Big\n s\mapsto \sum_{i=0}^{m(n)-1}1_{(t_i^n,t_{i+1}^n]}(s)(\psi'(X_{t_i^n\wedge
t})-\psi'(X_s))(g_s)\Big\n _{L^2([0,t];H)} = 0 \ \hbox{ in probability}.
$$
For this, in turn, it suffices to observe that $\P$-almost surely
\begin{align*}
\ &  \lim_{n\to\infty}  \Big\n \sum_{i=0}^{m(n)-1}1_{(t_i^n,t_{i+1}^n]}(s)(\psi'(X_{t_i^n\wedge
t})-\psi'(X_s))\Big\n _{L^\infty([0,t];E^*)} \\ & \qquad = 
\lim_{n\to\infty}   \sup_{0 \leq i\leq n-1}\sup_{s\in[t_i^n\wedge t,t_{i+1}^n\wedge t]}
\n  \psi'(X_{t_i^n\wedge t})-\psi'(X_s)\n  = 0
\end{align*}
by the path continuity of $X$.

\smallskip
{\em Step 2} -- In this step we prove the estimate \eqref{eq:remainder}.
By \eqref{Ap-lem-1-ep-1}, for all $x,y\in E$ and $r\in [0,1]$ we have
\begin{align*}          
 |\psi'(x+r(y-x))-\psi'(x)| \leq (\n x\n ^{p-2}\n x-y\n +\n x-y\n ^{p-1}).
\end{align*}  
Combining this with \eqref{eq:Taylor} we obtain  
\begin{align}\label{eq:twoterms}
| R_\psi(X_{t_{i}^n\wedge t},X_{t_{i+1}^n\wedge
t})| &\leq C\n X_{t_i^n\wedge t}\n ^{p-2} \n X_{t_{i+1}^n\wedge
t}-X_{t_i^n\wedge t}\n ^2+C \n X_{t_{i+1}^n\wedge t}-X_{t_i^n\wedge t}\n ^p.
\end{align}
We shall estimate the two terms on the right hand of \eqref{eq:twoterms} side separately.

For the first term,
using the inequality $|a+b|^2\leq 2|a|^2+2|b|^2$ we obtain
\begin{align*}
\ & \sum_{i=0}^{m(n)-1}\n X_{t_i^n\wedge
t}\n ^{p-2} \n X_{t_{i+1}^n\wedge t}-X_{t_i^n\wedge t}\n ^2 \\    
& \qquad \leq
 2\sum_{i=0}^{m(n)-1} \n X_{t_i^n\wedge t}\n ^{p-2} \Big\n  \int_{t_i^n\wedge
t}^{t_{i+1}^n\wedge t}a_s\,\d s   \Big\n ^2
+2\sum_{i=0}^{m(n)-1}\n X_{t_i^n\wedge t}\n ^{p-2}\Big\n \int_{t_i^n\wedge t}^{t_{i+1}^n\wedge t}g_s
\,\d W_s\Big\n ^2 =:I_1^n+I_2^n.
\end{align*}  
For the first term we have
\begin{align*}   
I_1^n & \leq  2C\sup_{s\in[0,t]}\n X_s\n ^{p-2} \times \sup_i\Big\n 
\int_{t_i^n\wedge t}^{t_{i+1}^n\wedge t}a_s\,\d s  
\Big\n \times\sum_{i=0}^{m(n)-1}\Big\n  \int_{t_i^n\wedge t}^{t_{i+1}^n\wedge
t}a_s\,\d s   \Big\n \\
   & \leq 2C\sup_{s\in[0,t]}\n X_s\n ^{p-2}\times \sup_i\Big\n  \int_{t_i^n\wedge
t}^{t_{i+1}^n\wedge t}a_s\,\d s   \Big\n \times
\int_0^t\n a_s\n \,\d s.   \\                        
\end{align*} 
By letting $n\rightarrow\infty$ we have
$\max_{0\le i\le m(n)-1} (t_{i+1}^n-t_i^n)\rightarrow0$, so 
$$\sup_{0\leq i\leq m(n)-1}\Big\n 
\int_{t_i^n\wedge t}^{t_{i+1}^n\wedge t}a_s\,\d s\Big\n \rightarrow 0$$ as
$n\rightarrow\infty$. Therefore, 
     \begin{align*}
                \lim_{n\rightarrow\infty}I_1^n=0,\
\mathbb{P}\text{-almost surely.}
                \end{align*}   
To estimate $I_2$ we use 
\eqref{Iso-ine} and Young's inequality with $\e>0$ to infer
\begin{align*}
 \E\, \liminf_n I_2^n \leq \liminf_n \, \E I_2^n
&=\liminf_n\, \E\sum_{i=0}^{m(n)-1}\n X_{t_i^n\wedge
t}\n ^{p-2}\Big\n \int_{t_i^n\wedge t}^{t_{i+1}^n\wedge t}g_s \,\d W_s\Big\n ^2\\
 &=\liminf_n\sum_{i=0}^{m(n)-1}\E\Big(\n X_{t_i^n\wedge
t}\n ^{p-2}\E\Big(\Big\n \int_{t_i^n\wedge t}^{t_{i+1}^n\wedge t}g_s
\,\d W_s\Big\n ^2\big|\F_{t_i^n\wedge t}\Big)\Big)\\  
 &\leq
C\liminf_n\sum_{i=0}^{m(n)-1}\E\Big(\n X_{t_i^n\wedge
t}\n ^{p-2}\E\Big(\int_{t_i^n\wedge t}^{t_{i+1}^n\wedge t}\n g_s\n ^2_{\gamma(H,E)}
\,\d s\big|\F_{t_i^n\wedge t}\Big)\Big)\\ 
  &\leq
C\liminf_n\sum_{i=0}^{m(n)-1}\E\Big(\n X_{t_i^n\wedge
t}\n ^{p-2}\int_{t_i^n\wedge t}^{t_{i+1}^n\wedge t}\n g_s\n ^2_{\gamma(H,E)}
\,\d s\Big)\\  
   &\leq
C\liminf_n\, \E\Big(\sup_{s\in[0,t]}\n X_s\n ^{p-2}\sum_{i=0}^{m(n)-1}\int_{
t_i^n\wedge t}^{t_{i+1}^n\wedge t}\n g_s\n ^2_{\gamma(H,E)} \,\d s\Big)\\  
   &=
C\E\Big(\sup_{s\in[0,t]}\n X_s\n ^{p-2}\int_{0}^{t}\n g_s\n ^2_{\gamma(H,E)}
\,\d s\Big)\\
& \le 
C\varepsilon\E\Big(\sup_{s\in[0,t]}\n X_s\n ^{p}\Big)
+ C\varepsilon^{1-\frac{p}{2}}\E  \Big(\int_{0}^{t}\n g_s\n ^2_{\gamma(H,E)}
\,\d s\Big)^{\frac{p}{2}}.
\end{align*} 

Next we estimate the second term in \eqref{eq:twoterms}. We have
\begin{align*}
\sum_{i=0}^{m(n)-1}\n X_{t_{i+1}^n\wedge
t}-X_{t_i^n\wedge t}\n ^p
&\leq C\sum_{i=0}^{m(n)-1}\Big\n  \int_{t_i^n\wedge
t}^{t_{i+1}^n\wedge t}a_s\,\d s   \Big\n ^p
+C\sum_{i=0}^{m(n)-1}\Big\n \int_{t_i^n\wedge
t}^{t_{i+1}^n\wedge t}g_s \,\d W_s\Big\n ^p =: I_3^n+I_4^n.
\end{align*}  
   A similar consideration as before yields 
\begin{align*}
 \lim_{n\rightarrow\infty} I_3^n\leq
C\lim_{n\rightarrow\infty}\sup_{0\leq i\leq m(n)-1}\Big\n  \int_{t_i^n\wedge
t}^{t_{i+1}^n\wedge t}a_s\,\d s   \Big\n ^{p-1}\times
\int_0^t\n a_s\n \,\d s=0.   
\end{align*}
  Moreover, by Burkholder's inequality \eqref{burk-ine}, 
\begin{align*}
    \E\liminf_n I_4^n 
\leq \liminf_n \E I_4^n 
& =C \liminf_n\sum_{i=0}^{m(n)-1}\E \Big\n \int_{t_i^n\wedge t}^{t_{i+1}^n\wedge
t}g_s \,\d W_s\Big\n ^p\\
 &\leq
C\liminf_n\sum_{i=0}^{m(n)-1}\E\Big(\int_{t_i^n\wedge
t}^{t_{i+1}^n\wedge t}\n g_s\n ^2_{\gamma(H,E)} \,\d s\Big)^{\frac{p}{2}}\\
&\leq C
\liminf_n\E\Big(\sum_{i=0}^{m(n)-1}\int_{t_i^n\wedge t}^{t_{i+1}^n\wedge
t}\n g_s\n ^2_{\gamma(H,E)} \,\d s\Big)^{\frac{p}{2}}\\
&=
C\E\Big(\int_0^t\n g_s\n ^2_{\gamma(H,E)}\,\d s\Big)^{\frac{p}{2}}.
\end{align*}

Collecting terms, for any $\e>0$ 
we obtain the estimate
\begin{align*}\E\liminf_{n\rightarrow\infty}\sum_{i=0}^{m(n)-1} | R_\psi(X_{t_{i}^n\wedge
t},X_{t_{i+1}^n\wedge t})|
 \le 
C\varepsilon\E\Big(\sup_{s\in[0,t]}\n X_s\n ^{p}\Big)
+ C(\varepsilon^{1-\frac{p}{2}}+1)\E  \Big(\int_{0}^{t}\n g_s\n ^2_{\gamma(H,E)}
\,\d s\Big)^{\frac{p}{2}}.
\end{align*}
\end{proof}

In the proof of Theorem \ref{thm:main} we will also need the following simple observation.

\begin{lem}\label{TH-main-rem-1}  
$\P$-Almost surely we have
\begin{align}\label{ep-90}
\liminf_{n\to\infty} \sup_{t\in [0,T]} \sum_{i=0}^{m(n)-1}
| R_\psi(X_{t_{i}^n\wedge t},X_{t_{i+1}^n\wedge
t})| \leq \liminf_{n\to\infty} \sum_{i=0}^{m(n)-1} | R_\psi(X_{t_{i}^n},X_{t_{i+1}^n})| .
\end{align} 
\end{lem}
\begin{proof}
Fix $t\in(0,T]$ and let $k(n)$ be the unique
index such that $t\in(t_{k(n)}^n,t_{k(n)+1}^n]$. 
Then 
      \begin{align*}  
\sum_{i=0}^{m(n)-1} | R_\psi(X_{t_{i}^n\wedge
t},X_{t_{i+1}^n\wedge t})| &=\sum_{i=0}^{k(n)-1} | R_\psi(X_{t_{i}^n},X_{t_{i+1}^n})| +
| R_\psi(X_{t_{k(n)}^n},X_{t})| \\
&\leq \!\! \sum_{i=0}^{m(n)-1}
| R_\psi(X_{t_{i}^n},X_{t_{i+1}^n})| + | R_\psi(X_{t_{k(n)}^n},X_{t})|   \\
&\leq \!\!\sum_{i=0}^{m(n)-1}
| R_\psi(X_{t_{i}^n},X_{t_{i+1}^n})| +C\n X_{t_{k(n)}^n}\n ^{p-2}
\n X_t-X_{t_{k(n)}^n}\n ^2+C \n X_t-X_{t_{k(n)}^n}\n ^p\\
&\leq \!\!
 \sum_{i=0}^{m(n)-1}
| R_\psi(X_{t_{i}^n},X_{t_{i+1}^n})| +C\sup_{s\in[0,T]}\n X_s\n ^{p-2}
\n X_t-X_{t_{k(n)}^n}\n ^2+C \n X_t-X_{t_{k(n)}^n}\n ^p.
\end{align*}
Now \eqref{ep-90} follows by
taking the limes inferior for $n\rightarrow\infty$ and using path continuity.
\end{proof}

\section{Proof of Theorem \ref{thm:main}}

We proceed in four steps. In Steps 1 and 2 we establish the estimate in the theorem for 
$g\in M^p([0,T];\gamma(H,E))$ with $2\le p<\infty$.
In order to be able to cover exponents $0<p<2$ in Step 3, we need a stopped version of the inequalities
proved in Steps 1 and 2. For reasons of economy of presentations, we therefore build in
a stopping time $\tau$ from the start. In Step 4 we finally consider the case where
$g\in M([0,T];\gamma(H,E))$.

We shall apply (a special case of) Lenglart's inequality \cite[Corollaire II]{[Lenglart]} which states that if 
$(\xi_t)_{t\in [0,T]}$ and
$(a_t)_{t\in [0,T]}$ are continuous non-negative adapted processes, the latter non-decreasing, 
such that $\E \xi_\tau \le \E a_\tau$ for all stopping times $\tau$ with values in $[0,T]$, 
then for all $0<r<1$ one has
\begin{equation}\label{eq:Lenglart} 
\E \sup_{0\le t\le T} \xi_t^r \le \frac{2-r}{1-r}\E a_T^r.
\end{equation}

\smallskip
{\em Step 1} -- 
Fix $p\ge 2$ and suppose first that $g\in M^p([0,T];\gamma(H,\dom(A))$.  
As is well known (see \cite{[Ondrejat]}), under this condition
the process $X_t=\int_0^t e^{(t-s)A}
g_s \,\d W_s$ is a strong solution to the equation
\begin{align*}
dX_t=AX_t\, \d t+g_t\,\d W_t,\ \ t\ge 0; \  \ X_0=0.
\end{align*}
   In other words, $X$ satisfies  
\begin{align*}
   X_t=\int_0^tAX_s\,\d s+\int_0^tg_s \,\d W_s\ \ \forall t\in [0,T] \ \ \mathbb{P}\text{-almost surely}.
\end{align*}
Hence if $\tau$ is a stopping time with values in $[0,T]$,
then by \eqref{eq:stopping},
\begin{align*}
   X_{t\wedge \tau} =\int_0^t 1_{[0,\tau]}(s)AX_{s}\,\d s+\int_0^t  1_{[0,\tau]}(s)  g_s \,\d W_s
\ \ \forall t\in [0,T], \ \ \mathbb{P}\text{-almost surely}.
\end{align*}
Let us check next that  $a_t := 1_{[0,\tau]}(t) AX_t$ 
satisfies the assumptions of Theorem \ref{thm-main-1}.
Indeed,  with $h_t:=  1_{[0,\tau]}(t) Ag_t$ we have, 
using the contractivity of the semigroup $S$ and Burkholder's inequality \eqref{burk-ine},
\begin{align*}
 \E \Big(\int_0^T \n a_t\n \,\d t \Big)^p 
 & \le
\E \Big(\int_0^T \Big\n \int_0^t e^{(t-s)A}h_s\,\d W_s\Big\n \,\d t \Big)^p 
\\ & \le C T^{p-1} \E \int_0^T \Big\n \int_0^t e^{(t-s)A}h_s\,\d W_s\Big\n ^p\,\d t
\\ & \le C T^{p-1} \E \int_0^T \Big(\int_0^t \n e^{(t-s)A}h_s\n_{\gamma(H,E)} ^2\,\d s\Big)^\frac{p}{2} \,\d t
\\ & \le C T^p\E \Big(\int_0^T \n h_s\n_{\gamma(H,E)} ^2\,\d s\Big)^\frac{p}{2} < \infty.
\end{align*}
Hence we may apply Theorem \ref{thm-main-1} and infer that
\begin{align*}
 \n X_{t\wedge \tau}\n ^p & 
=\int_0^t 1_{[0,\tau]}(s)\psi'(X_s)(AX_s)\,\d s
\\ & \qquad +\int_0^t 1_{[0,\tau]}(s)\psi'(X_s)( g_s) \,\d W_s
+\lim_{n\to\infty}\sum_{i=0}^{m(n)-1} R_\psi(X_{t_{i}^n\wedge t\wedge
\tau},X_{t_{i+1}^n\wedge t\wedge \tau}) 
\\ & \le
\int_0^t 1_{[0,\tau]}(s)\psi'(X_s)( g_s) \,\d W_s
 +\lim_{n\to\infty} \sum_{i=0}^{m(n)-1}  R_\psi(X_{t_{i}^n\wedge t\wedge
\tau},X_{t_{i+1}^n\wedge t\wedge \tau}) 
\end{align*}
since $\psi'(x)(Ax)\leq 0$ for all $x\in E$ 
by the contractivity of $e^{tA}$ (see \cite[Lemma 4.2]{BHZ}). 
Hence, by Lemma \ref{TH-main-rem-1},
\begin{align*}
 \E\sup_{t\in[0,T]}\n X_{t\wedge \tau}\n ^p
&   \leq\E\sup_{t\in[0,T]}
\int_0^t 1_{[0,\tau]}(s)\psi'(X_s)( g_s) \,\d W_s
 +\E\sup_{t\in [0,T]}\liminf_{n\rightarrow\infty}\sum_{i=0}^{m(n)-1} | R_\psi(X_{t_{i}^n\wedge t\wedge
\tau},X_{t_{i+1}^n\wedge t\wedge \tau})| \\
&  \leq\E\sup_{t\in[0,T]}
\int_0^t 1_{[0,\tau]}(s)\psi'(X_s)( g_s) \,\d W_s
 +\E\liminf_{n\rightarrow\infty}\sum_{i=0}^{m(n)-1} | R_\psi(X_{t_{i}^n\wedge
\tau},X_{t_{i+1}^n\wedge \tau})| \\
& \leq C\E\sup_{t\in[0,T]}
\int_0^t1_{[0,\tau]}(s)\psi'(X_s)( g_s) \,\d W_s
\\ & \qquad +\varepsilon
C\E\sup_{s\in[0,T]} \n X_{s\wedge \tau}\n^p +C_{\e}\E\Big(\int_{0}
^T1_{[0,\tau]}(s)\n g_s\n _{\gamma(H,E)}^2\,\d s\Big)^{\frac{p}{2}}.
\end{align*}
By Burkholder's inequality \eqref{burk-ine}
and the identity $\n \psi'(y)\n =p\n y\n ^{p-1}$,
\begin{align*} 
\E\sup_{t\in[0,T]}\Big|\int_0^t1_{[0,\tau]}(s)\psi'(X_s)( g_s)
\,\d W_s\Big|
    &\leq
C\E\Big(\int_0^T1_{[0,\tau]}(s)\big\n \psi'(X_s)\big\n ^2\n g_s\n _{\gamma(H,E)}^2\,\d s\Big)^{
\frac{1}{2}}\\
&=C
\E\Big(\int_0^T1_{[0,\tau]}(s)\big\n X_s\big\n ^{2(p-1)}\n g_s\n _{\gamma(H,E)}^2\,\d s\Big)^{\frac{1}{2
}}\\
&\leq C
\E\Big(\sup_{t\in[0,T]}\n X_{t\wedge \tau}\n ^{p-1}\Big(\int_0^T1_{[0,\tau]}(s)\n g_s\n _{\gamma(H,E)}^2\,
\d s\Big)^{\frac{1}{2}}\Big)\\
&\leq
C p^p\Big(\E\sup_{t\in[0,T]}\n X_{t\wedge \tau}\n ^{p}\Big)^{\frac{p-1}{p}}\Big(
\E\Big(\int_0^T1_{[0,\tau]}(s)\n g_s\n _{\gamma(H,E)}^2\,\d s\Big)^{\frac{p}{2}}\Big)^{\frac{1}{p}}\\
&\leq
C\varepsilon\E\sup_{t\in[0,T]}\n X_{t\wedge \tau}\n ^{p}+ C_\e
\E\Big(\int_0^T1_{[0,\tau]}(s)\n g_s\n _{\gamma(H,E)}^2\,\d s\Big)^{\frac{p}{2}},
    \end{align*}                                                                                                    
where we also used the H\"older's inequality and Young's
inequality. 

Combining these estimates and taking $\varepsilon>0$ small enough, 
we infer that
 \begin{align*}
   \E\sup_{t\in[0,T]}\n X_{t\wedge \tau}\n ^p\leq
C\E\Big(\int_{0}^T1_{[0,\tau]}(s)\n g_s\n _{\gamma(H,E)}^2\,\d s\Big)^{\frac{p}{2}}. 
 \end{align*} 

\smallskip
{\em Step 2} -- Now let $g\in M^p([0,T];\gamma(H,E)$ be arbitrary.
Set $g^n=n(nI-A)^{-1} g$, $n\ge 1$. 
These processes satisfy the assumptions of Step 1 
and we have $\n g^n\n _{\gamma(H,E)}\leq \n g\n _{\gamma(H,E)}$ pointwise.
Define $X_t^n=\int_0^te^{(t-s)A}g_s^n\,\d s.$
From Step 1 we know that for any stopping time $\tau$ in $[0,T]$ we have
\begin{align*}  
\E\sup_{t\in[0,T]}\n X_{t\wedge\tau}^n\n ^p\leq
C\E\Big(\int_{0}^T1_{[0,\tau]}(s)\n g_s^n\n _{\gamma(H,E)}^2\,\d s\Big)^{\frac{p}{2}}.
\end{align*}
In particular, as $n,m\rightarrow\infty$, 
\begin{align*} 
       \E\sup_{t\in[0,T]}\n X_{t}^n-X_{t}^m\n ^p\rightarrow 0.
\end{align*} 
In these circumstances there is a process $\bar{X}$ such that
$\lim_{n\rightarrow\infty}\E\sup_{t\in[0,T]}\n \bar{X}_{t}^n-X_{t}\n ^p= 0$ and
\begin{align}\label{lem:ineq-stop}
  \E\sup_{t\in[0,T]}\n \bar{X}_{t\wedge \tau}\n ^p\leq
C\E\Big(\int_{0}^T1_{[0,\tau]}(s)\n g_s\n _{\gamma(H,E)}^2\,\d s\Big)^{\frac{p}{2}}.  
\end{align}
Also, notice that for every $t\in[0,T]$, we have 
\begin{align*} 
\E\n X_t^n- X_t\n ^p=\E\Big\n \int_0^te^{(t-s)A}g_s^n\,\d s-\int_0^te^{(t-s)A}
g_s\,\d s\Big\n ^p\leq C\big(\E\int_0^t\n g_s^n-g_s\n _{\gamma(H,E)}^2\,\d s\Big)^p.
\end{align*}
Hence $X_t^n\rightarrow X_t$ in $L^p(\Omega;E)$. Therefore, $\bar{X}$ is a
modification of $X$.  
This concludes the proof for $p\ge 2$.

\smallskip
{\em Step 3} -- In this step we extend the result to exponents $0<p<2$.
First consider the case where $g\in M^2([0,T];\gamma(H,E))$.
By \eqref{lem:ineq-stop}, for all stopping times $\tau$ in $[0,T]$ we have   
$$\E \n X_\tau\n ^2 \le C\E\int_0^\tau \n g_s\n ^2_{\gamma(H,E)}\, \d s.$$
It then follows from Lenglart's inequality \eqref{eq:Lenglart} 
that for all $0<p<2$,
$$\E \sup_{t\in [0,T]}\n X_t\n ^p \le C \E\Big(\int_0^T \n g_s\n ^2_{\gamma(H,E)}\, \d s\Big)^\frac{p}{2}.
$$
For $g\in M^p([0,T];\gamma(H,E))$ the result follows by approximation.

{\em Step 4} -- Finally, the existence of a continuous version for the process $X$ under the assumption 
$g\in M([0,T];\gamma(H,E))$ follows by a standard localisation argument.



\begin{thebibliography}{99} 

\bibitem{[Ana]} G.A. Anastassiou, S.S. Dragomir, 
\textit{On some estimates of the remainder in Taylor's formula}, 
J. Math. Anal. Appl. \textbf{263}, no. 1: 246-263, 2001.

\bibitem{[Ass]} P. Assouad, 
\textit{Espaces $p$-lisses et $p$-convexes, in\'egalit\'es de Burkholder}, 
S\'em. Maurey-Schwartz 1974-1975, ``Espaces $L^p$, applications radonifiantes, 
g\'eometrie des espaces de Banach'', 
Expos\'e XV, \'Ecole polytechnique, Centre de Math\'ematiques, Paris, 1975.
 
\bibitem{BHZ}  Z. Brze\'{z}niak, E. Hausenblas, J. Zhu,
\textit{Maximal inequality of stochastic convolution driven by compensated Poisson random measures 
in Banach spaces}, arXiv:1005.1600.

\bibitem{[BrzPes]} Z. Brze\'{z}niak, S. Peszat, 
\textit{Maximal inequalities and exponential estimates for stochastic 
convolutions in Banach spaces}, 
in: ``Stochastic processes, physics and geometry: new interplays, I'' (Leipzig, 1999), 55-64,
CMS Conf. Proc., 28, Amer. Math. Soc., Providence, RI, 2000. 

\bibitem{[DZ]} G. Da Prato, J. Zabczyk, 
\textit{A note on stochastic convolutions}, 
Stochastic Anal. Appl. \textbf{10}: 143-153, 1992. 

\bibitem{[Dett]} E. Dettweiler, 
\textit{Stochastic integration relative to 
Brownian motion on a general Banach space}, 
Doga-Tr. J. of Mathematics \textbf{15}: 6-44, 1991.

\bibitem{[Deville]}R. Deville, G. Godefroy, V. Zizler, 
 ``Smoothness and renormings in Banach spaces'',
 Pitman Monographs and Surveys in Pure and Applied Mathematics, 64,
Longman Scientific \& Technical, Harlow; copublished in the United States with
John Wiley \& Sons, Inc., New York, 1993.   

\bibitem{[Fig]} T. Figiel, 
\textit{On the moduli of convexity and smoothness},
Studia Math. \textbf{56}: 121-155, 1976

\bibitem{[HauSei]} E. Hausenblas, J. Seidler, 
\textit{A note on maximal inequality for stochastic convolutions}, 
Czechoslovak Math J., \textbf {51} (124) (4): 785-790, 2001.

\bibitem{[Ich]} A. Ichikawa, 
\textit{Some inequalities for martingales and stochastic convolutions}, 
Stochastic Anal. Appl. \textbf{4}, no. 3: 329-339, 1998

\bibitem{[Kot1]} P. Kotelenez, 
\textit{A submartingale type inequality with applications to stochastic convolution equations}, 
Stochastics \textbf{8}: 139-151, 1982.

\bibitem{[Kot2]} P. Kotelenez, 
\textit{A stopped Doob inequality for stochastic convolution integrals and stochastic evolution equations}, 
Stochastic Anal. \textbf{2}: 245-265, 1984.  

\bibitem{[Lenglart]} E. Lenglart,
\textit{Relation de domination entre deux processus}, 
Ann. Inst. H. Poincar\'e Sect. B (N.S.) \textbf{13}, no. 2: 171-179, 1977.

\bibitem{[LeoSun]} I.E. Leonard, K. Sundaresan, 
\textit{Geometry of Lebesgue-Bochner function spaces - smoothness},
Trans. Amer. Math. Soc. \textbf{198}: 229-251, 1974.
 
\bibitem{[Neidhardt]} A.L. Neidhardt, 
``Stochastic integrals in 2-uniformly smooth Banach spaces'', 
Ph.D Thesis, University of Wisconsin, 1978. 

\bibitem{[Ondrejat]} M. Ondrej\'{a}t, 
\textit{Uniqueness for stochastic evolution equations in Banach spaces}, 
Dissertationes Math. (Rozprawy Mat.) \textbf{426}, 2004.

\bibitem{[P1]} G. Pisier, 
\textit{Martingales with values in uniformly convex spaces}, 
Israel J. Math \textbf{20}: 326-350, 1975.

\bibitem{[Sei]} J. Seidler, 
\textit{Exponential estimates for stochastic convolutions in 2-smooth Banach spaces},
Electron. J. Probab. \textbf{15}, no. 50: 1556-1573, 2010. 

\bibitem{[Tub]} L. Tubaro, 
\textit{An estimate of Burkholder type for stochastic processes defined by the stochastic integral}, 
Stochastic Anal, Appl. \textbf{2}: 187-192, 1984.
   
\bibitem{[VerWei]} M.C. Veraar, L.W. Weis,
\textit{A note on maximal estimates for stochastic convolutions}, 
Czechoslovak Math. J. \textbf{61}, no. 3: 743-758, 2011. 

\end{thebibliography}
\end{document}